\documentclass[12pt,letter]{article}
\usepackage[utf8]{inputenc}
\usepackage[width=16cm,height=21cm]{geometry}
\usepackage{comment}
\usepackage{cite}

\usepackage{url}

\usepackage[colorlinks,allcolors=blue]{hyperref}
\newcommand{\doi}[1]{\href{https://doi.org/#1}{\url{doi:#1}}}
\newcommand{\myurl}[1]{\href{#1}{\url{#1}}}
\newcommand{\journaltitle}[1]{#1}
\newcommand{\booktitle}[1]{#1}
\newcommand{\myvolume}[1]{\textbf{#1}}
\newcommand{\volumeyearpages}[3]{\textbf{#1}, #3 (#2)}
\newcommand{\volumeissueyearpages}[4]{\textbf{#1}(#2), #4 (#3)}

\usepackage{amsmath,amssymb}
\newcommand{\al}{\alpha}
\newcommand{\be}{\beta}
\newcommand{\bB}{\mathbb{B}}
\newcommand{\Ber}{\operatorname{Ber}}
\newcommand{\bC}{\mathbb{C}}
\newcommand{\bD}{\mathbb{D}}
\newcommand{\bH}{\mathbb{H}}
\newcommand{\bN}{\mathbb{N}}
\newcommand{\bNz}{\mathbb{N}_0}
\newcommand{\bR}{\mathbb{R}}
\newcommand{\bS}{\mathbb{S}}

\newcommand{\cA}{\mathcal{A}}
\newcommand{\cB}{\mathcal{B}}
\newcommand{\dif}{\mathrm{d}}
\newcommand{\enumber}{\operatorname{e}}
\newcommand{\iu}{\mathrm{i}}
\newcommand{\de}{\delta}
\newcommand{\Ga}{\Gamma}
\newcommand{\Mob}{\operatorname{M\ddot{o}b}}
\renewcommand{\phi}{\varphi}
\newcommand{\Om}{\Omega}

\renewcommand{\Im}{\operatorname{Im}}
\renewcommand{\Re}{\operatorname{Re}}
\newcommand{\dpar}{\partial}
\newcommand{\wt}[1]{\widetilde{#1}}
\newcommand{\conju}[1]{\overline{#1}}

\newcommand{\Be}{\operatorname{B}}
\newcommand{\cK}{\mathcal{K}}
\newcommand{\cL}{\mathcal{L}}

\usepackage{mathtools}
\newcommand{\eqdef}{\coloneqq}

\usepackage{amsthm}
\newtheorem{theorem}{Theorem}[section]
\newtheorem{proposition}[theorem]{Proposition}
\newtheorem{lemma}[theorem]{Lemma}
\newtheorem{corollary}[theorem]{Corollary}

\theoremstyle{definition}
\newtheorem{definition}[theorem]{Definition}

\newtheorem{remark}[theorem]{Remark}
\numberwithin{equation}{section}

\usepackage{xcolor}

\title{\texorpdfstring{\vspace{-2.5em}Homogeneously polyanalytic kernels\\
on the unit ball and the Siegel domain}%
{Homogeneously polyanalytic kernels on the unit ball and
the Siegel domain}}

\author{Christian Rene Leal-Pacheco,
Egor A. Maximenko,\\
Gerardo Ramos-Vazquez}

\begin{document}

\maketitle

\begin{abstract}
We prove that the homogeneously polyanalytic functions of total order $m$,
defined by the system of equations
$\overline{D}^{(k_1,\ldots,k_n)} f=0$
with $k_1+\cdots+k_n=m$,
can be written as polynomials
of total degree $<m$ in variables $\overline{z_1},\ldots,\overline{z_n}$,
with some analytic coefficients.
We establish a weighted mean value property for such functions, using a reproducing property of Jacobi polynomials.
After that, we give a general recipe to transform a reproducing kernel by a weighted change of variables.
Applying these tools, we compute the reproducing kernel of the Bergman space of homogeneously polyanalytic functions on the unit ball in $\mathbb{C}^n$ and on the Siegel domain.
For the one-dimensional case,
analogous results were obtained
by Koshelev (1977),
Pessoa (2014),
Hachadi and Youssfi (2019).









\end{abstract}

\tableofcontents

\clearpage

\section{Introduction}

Bergman~\cite{Bergman1950} comprehensively studied spaces of square-integrable analytic functions on one-dimensional domains,
considering them as reproducing kernel Hilbert spaces (RKHS).
For some of multidimensional generalizations, see
\cite{FarautKoranyi1990,Vukotic1993,Zhu2005}.
Polyanalytic functions, have been attracted attention of many mathematicians since the beginning of the 20th century.
See some of their properties, applications, and history in~\cite{Abreu2010,AbreuFeichtinger2014,Balk1991,Fedorovsky2016,HaimiHedenmalm2013}.

Koshelev~\cite{Koshelev1977} proved that every integrable $m$-analytic function $f$ on $\bD$ fulfills an analog of the mean value property:
\[
f(0)
= \frac{1}{\pi}
\int_{\bD} f(z)\,P(|z|^2)\,
\dif\mu(z),
\]
where $P$ is a certain polynomial of degree $m-1$ with explicitly computed coefficients. Furthermore, he proved that
the corresponding space $\cA_m^2(\bD)$
is a RKHS and gave an explicit formula for the reproducing kernel (RK) at the arbitrary point $z_0$ of the disk, using the M\"{o}bius transormation $\phi_{z_0}$
that interchanges $z_0$ with the origin.
Due to the format of the journal, his explanation was extremely short: ``although the class of polyanalytic functions is not invariant relative to fractional-linear transformations, this device is still usefull thanks to the presence of $K_n(z,z_0)$ under the integral sign''.
Pessoa~\cite{Pessoa2014} identified $P$ with a certain shifted Jacobi polynomial and explained very clearly, how to translate the reproducing property from the origin to an arbitrary point $z_0$ of the disk.
Namely, he found a correcting factor that restores the polyanalyticity and
converts the composition operator $f\mapsto f\circ \phi_{z_0}$ into a unitary operator in $\cA_m^2(\bD)$.
He also computed~\cite{Pessoa2013}
the RK of the space $\cA_m^2(\bH_1)$ of $m$-analytic functions on the upper halfplane $\bH_1$ in $\bC$.
Hachadi and Youssfi~\cite{HachadiYoussfi2019} studied polyanalytic functions on the disk and on the entire complex plane, provided with radial measures.
In particular, they computed the RK of $\cA_m^2(\bD,\mu_\al)$, where $\dif{}\mu_\al(z)=\frac{1}{\pi}(1-|z|^2)^{\al}\,\dif\mu(z)$.

In this paper, we extend some of these results to the unit ball $\bB_n$ in $\bC^n$
and to the Siegel domain
$\bH_n\eqdef\{\xi\in\bC^n\colon\ \Im(\xi_n)>|\xi_1|^2+\dots+|\xi_n|^2\}$.

Let $\bN=\{1,2,\ldots\}$, $\bNz=\{0,1,2,\ldots\}$,
$n\in\bN$.
We employ the usual notation for the multi-indices and the notation $|\cdot|$
for the norm in $\bC^n$, see~\cite[Section~1.1]{Rudin1980}.
Given an open set $\Om$ in $\bC^n$,
a multi-index $k=(k_1,\ldots,k_n)$
in $\bNz^n$
and a function of the class $C^{|k|}(\Om)$,
we denote by $\conju{D}^k f$ the Wirtinger derivative of $f$ of the order $k$ (such derivatives were previously used by Poincar\'{e}, Pompeiu, and Kolossov).
In a more classical notation,
\[
\conju{D}^k f(z) \eqdef
\frac{\dpar^{|k|}}{\dpar^{k_1}\overline{z_1} \cdots
\dpar^{k_n}\overline{z_n}}\,f(z) \qquad (z\in\Omega).
\]
Let $\cA(\Om)$ be the class of all analytic functions on $\Om$.
It is defined by the system of equations
\begin{equation}\label{eq:analytic}
\conju{D}^{(1,0,\ldots,0)}f=0,\qquad
\conju{D}^{(0,1,\ldots,0)}f=0,\qquad
\ldots,\qquad
\conju{D}^{(0,0,\ldots,1)}f=0.
\end{equation}
Given an open subset $\Omega$ of $\bC^n$
and a multi-index $k=(k_1,\ldots,k_n)$ in $\bN^n$,
$k$-analytic functions on $\Om$
are defined~\cite[Section~6.4]{Balk1991} as functions that can be represented in the form
\begin{equation}\label{eq:Balk}
f(z)
=\sum_{j_1,\ldots,j_n=0}^{k_1-1,\ldots,k_n-1}
g_j(z)\,\conju{z}^j,
\end{equation}
where all functions $g_j$ are analytic.
We denote by $\cA_k(\Om)$ the class of all functions of the form~\eqref{eq:Balk}.
For simply connected domains $\Om$,
such functions can also be characterized
as smooth solutions of
the system of differential equations
\begin{equation}\label{eq:Balk1}
\conju{D}^{(k_1,0,\ldots,0)}f=0,\qquad
\conju{D}^{(0,k_2,\ldots,0)}f=0,\qquad
\ldots,\qquad
\conju{D}^{(0,0,\ldots,k_n)}f=0.
\end{equation}
Instead of considering polyanalytic functions of a given multi-order $k$, we prefer to work with the following classes of ``homogeneously polyanalytic'' functions.

\begin{definition}\label{def:Am}
Let $\Om$ be an open set in $\bC^n$ and $m\in\bN$.
We say that $f\colon\Om\to\bC$
is \emph{homogeneously polyanalytic of total order $m$}
or just \emph{$m$-analytic},
if $f$ belongs to the class $C^m(\Om)$
and $\conju{D}^k f=0$ for every $k$ in $\bN_0^n$
with $|k|=m$.
We denote by $\cA_m(\Om)$
the set of all such functions.
\end{definition}

The multi-indices $k$ with $|k|=m$ can be associated with $m$-multisubsets of the set $\{1,\ldots,n\}$, and the number of such multi-indices is $\binom{n+m-1}{m}$.
For example, the class $\cA_1(\Omega)=\cA(\Omega)$
is defined by $n$ differential equations~\eqref{eq:analytic}.

\begin{definition}\label{def:wtAm}
Let $\Om$ be an open set in $\bC^n$ and $m\in\bN$.
We denote by $\wt{\cA}_m(\Om)$ the set of all functions $f\colon\Om\to\bC$ that can be written in the form
\begin{equation}\label{eq:polyanalytic_function_as_polynomial}
f(z)
=\sum_{|j|<m} 
h_j(z)\overline{z}^j,
\end{equation}
where $h_j\in\cA(\Om)$
for all $j\in\bNz^n$ with $|j|<m$.
\end{definition}

In Section~\ref{sec:general_form}
we prove that
$\wt{\cA}_m(\Om)=\cA_m(\Om)$.
Obviously, $\cA_m(\Om)$
is a complex vector space.
In Proposition~\ref{prop:Am_is_invariant_under_linear_changes_of_variables}
we show that the space $\cA_m(\Om)$, with $m$ in $\bN$,
is invariant under linear changes of variables (of course, the domain can change).
In a contrast, the spaces $\cA_k$, with $n\ge 2$ and $k\in\bN^n$, $k\ne(1,1,\ldots,1)$, are not invariant under linear changes of variables; see Proposition~\ref{prop:Ak_is_not_invariant_under_linear_changes_of_variables}.
If $m\in\bN$, then $\cA_m(\Om)\subseteq\cA_{(m,\ldots,m)}(\Om)$, and some results about $k$-analytic functions ($k\in\bN^n$) can be applied to $\cA_m(\Om)$.
On the other hand,
if $k$ in $\bN^n$, then
$\cA_k(\Om)\subseteq\cA_{|k|+1-n}(\Om)$.

From now on, we denote by $\mu$
the Lebesgue measure on $\bC^n$.

\begin{definition}\label{def:Bergman_polyanalytic}
Let $\Om$ be an open set in $\bC^n$, $m\in\bN$,
$W\colon\Om\to(0,+\infty)$ be a continuous function,
and $\dif\nu=W\,\dif\mu$.
We denote by $\cA_m^2(\Om,\nu)$
the set of all functions $f\in\cA_m(\Om)$
that are square-integrable with respect to $\nu$.
We consider this space with the inner product inherited from $L^2(\Om,\nu)$.
Furthermore, we denote by $\cA_{(m)}^2(\Om,\nu)$
the orthogonal complement of $\cA_{m-1}^2(\Om,\nu)$ in $\cA_m(\Om,\nu)$.
Here $\cA_0^2(\Om,\nu)\eqdef\{0\}$.
\end{definition}

Section~\ref{sec:weighted_mean_value_property} contains a weighted mean-value property for integrable functions belonging to $\cA_m(\Om)$.
As a consequence of this property,
$\cA_m^2(\Om,\nu)$ is a RKHS.
In Section~\ref{sec:pushforward_RK}
we show how the RK transforms under a weighted change of variables.
In Section~\ref{sec:RK_unit_ball}
we use the previous tools
to compute the RK of $\cA_m^2(\bB_n,\mu_\al)$,
where $\bB_n$ is the unit ball in $\bC^n$ and $\mu_\al$ is the standard radial measure on $\bB_n$, see~\eqref{eq:mu_alpha}.
Finally, in Section~\ref{sec:RK_Siegel}
we compute the RK of $\cA_m^2(\bH_n,\nu_\al)$,
where $\bH_n$ is the standard Siegel domain in $\bC^n$
and $\nu_\al$ is a weighted Lebesgue measure (see~\eqref{Siegel} and~\eqref{eq:measure_on_Siegel}).

There are many recent investigations on Toeplitz operators, acting in polyanalytic Bergman spaces over one-dimensional domains
\cite{CuckovicLe2012,
HutnikHutnikova2015,
LoaizaRamirez2017,
RamirezSanchez2015,
Vasilevski2008book,
MaximenkoTelleria2020}.
We hope that this paper can serve as a basis for some multidimensional generalizations,
see Remarks~\ref{rem:rotations_in_the_ball},
\ref{rem:biholomorphic_changes_of_variables_in_the_ball},
and
\ref{rem:biholomophic_changes_of_variables_in_Siegel}.

Finalizing this introduction,
we mention several multidimensional results
about polyanalytic spaces and kernels
in other settings.
Askour, Intissar, and Mouayn~\cite{AskourIntissarMouayn1997} computed the RK of the
space of polyanalytic functions on $\bC^n$, square-integrable with respect to the Gaussian weight (i.e.,
the polyanalytic Bargmann--Segal--Fock space).
If $k\in\bN^n$ and $(\Om,\nu)$ is a direct product of one-dimensional domains with some weights
(for example, $\Om=\bC^n$ or $\Om=\bD^n$),
then the RK of $\cA_k(\Om,\nu)$ can be obtained as the tensor product of the corresponding reproducing kernels on one-dimensional domains~\cite{HachadiYoussfi2019}.
Ram\'{i}rez Ortega and S\'{a}nchez Nungaray~\cite{RamirezSanchez2013}
defined some polyanalytic-type spaces
on the Siegel domain $\bH_n$
by other systems of differential equations,
involving non-constant coefficients.

\section{Homogeneously polyanalytic functions}
\label{sec:general_form}

Let $\Om$ be an open set in $\bC^n$ and $m\in\bN$.
In this section we show that 
$\cA_m(\Om) = \wt{\cA}_m(\Om)$
and mention some other properties of $\cA_m(\Om)$.

\begin{lemma}\label{lem:g_is_analytic}
Let $f\in\cA_m(\Omega)$.
Then the following function is analytic:
\[
g(z) \eqdef
\sum_{\substack{k\in\bNz^n\\|k|<m}}
\frac{(-1)^{|k|}}{k!} (\overline{D}^k f)(z)\, \overline{z}^k.
\]
\end{lemma}

\begin{proof}
Let $p\in\{1,\ldots,n\}$ and $e_p$ be the $p$-th canonical vector in $\bNz^n$,
i.e., $e_p\eqdef(\de_{p,s})_{s=1}^n$, where $\de$ is the Kronecker's delta.
We have to show that $\conju{D}^{e_p}g=0$.
By the product rule,
\[
(\overline{D}^{e_p}g)(z) = S_1(z) + S_2(z) + S_3(z) + S_4(z),
\]
where
\begin{alignat*}{2}
S_1(z) &= \sum_{|k|<m-1} \tfrac{(-1)^{|k|}}{k!} (\overline{D}^{k+e_p} f)(z) \overline{z}^k, &\qquad S_2(z) &= \sum_{|k|=m-1} \tfrac{(-1)^{|k|}}{k!} (\overline{D}^{k+e_p} f)(z) \overline{z}^k, \\
S_3(z) &= \sum_{\substack{|k|<m \\ k_p=0}} \tfrac{(-1)^{|k|}}{k!} (\overline{D}^k f)(z) \overline{D}^{e_p}(\overline{z}^k), 
& S_4(z) &= \sum_{\substack{|k|<m\\ k_p>0}} \tfrac{(-1)^{|k|}}{(k-e_p)!} (\overline{D}^k f)(z) \overline{z}^{k-e_p}.
\end{alignat*}
We have that $S_2(z)=0$,
because $f \in \cA_m(\Omega)$
and $|k+e_p|=m$ in the sum defining $S_2$.
Also $S_3(z)=0$, because
$\conju{D}^{e_p}\conju{z}^k=0$ when $k_p=0$.
Finally, 
with the change of variable $j=k-e_p$,
we rewrite $S_4(z)$ as
\[
S_4(z) = - \sum_{|j|<m-1} \frac{(-1)^{|j|}}{j!} (\overline{D}^{r+e_p} f)(z)\,\overline{z}^j.
\]
Therefore, $(\conju{D}^{e_p} g)(z) = S_1(z) + S_4(z) = 0$.
\end{proof}

\begin{lemma}\label{lem:f_in_A_p_g_in_A_q_implies_fg_in_A_p+q-1}
Let $f \in \wt{\cA}_p(\Omega)$ and $g \in \wt{\cA}_q(\Omega)$. Then $fg \in \wt{\cA}_{p+q-1}(\Omega)$.
\end{lemma}

\begin{proof}
This lemma follows from the elementary observation that if $j\in\bNz^n$ and $k\in\bNz^n$,
with $|j|<p$ and $|k|<q$,
then $\conju{z}^j\,\conju{z}^k=\conju{z}^{j+k}$
and $|j+k|=|j|+|k|<p+q-1$.
\end{proof}

\begin{theorem}\label{thm:Am_characterization}
Let $\Om$ be an open set in $\bC^n$ and $m\in\bN$.
Then $\cA_m(\Omega)=\wt{\cA}_m(\Omega)$.
\end{theorem}

\begin{proof}
It is well known that $\cA(\Omega) = \wt{\cA}(\Omega)$. Let $m>1$.
It is obvious that $\wt{\cA}_m(\Omega) \subseteq \cA_m(\Omega)$.
We show, by induction on $m$, that 
$\cA_m(\Omega) \subseteq \wt{\cA}_m(\Omega)$. 
Suppose $\cA_p(\Omega) \subseteq \wt{\cA}_p(\Omega)$ for every $p<m$ and
let $f\in \cA_m(\Omega)$. Define $g$ as in 
Lemma~\ref{lem:g_is_analytic}, then observe that
\[
f(z) = - \sum_{0<|k|<m} \frac{(-1)^{|k|}}{k!} (\overline{D}^k f)(z) \, \overline{z}^k + g(z).
\]
For every $k$ with $0<|k|<m$,
we have
$\conju{z}^k \in \wt{\cA}_{|k|+1}(\Om)$
and
$\conju{D}^k f \in \cA_{m-|k|}(\Om) \subseteq \wt{\cA}_{m-|k|}(\Om)$;
the last inclusion holds by the induction hypothesis.
Finally, apply Lemma~\ref{lem:f_in_A_p_g_in_A_q_implies_fg_in_A_p+q-1}.
\end{proof}

\begin{corollary}\label{cor:Am_around_point_a}
Let $f\in\cA_m(\Om)$ and $a\in\Om$.
Then there exists a family of functions $(h_k)_{|k|<m}$
in $\cA(\Om)$,
such that for every $z$ in $\Om$,
\begin{equation}\label{eq:Am_around_point_a}
f(z)=\sum_{\substack{k\in\bNz^n\\|k|<m}}
h_k(z)\,(\conju{z}-\conju{a})^k.
\end{equation}
\end{corollary}

\begin{proof}
First, we write $f$ as~\eqref{eq:polyanalytic_function_as_polynomial}.
Then, expanding $\conju{z}^j=(\conju{z}-\conju{a}+\conju{a})^j$ into multi-powers of $\conju{z}-\conju{a}$ and regrouping the summands, we obtain a sum of the form~\eqref{eq:Am_around_point_a}.
\end{proof}

\begin{corollary}
\label{cor:Am_into_power_series}
Let $f\in\cA_m(\Om)$, $a\in\Om$, and $r>0$ such that $a+r\bB_n\subseteq\Om$.
Then there exists a family $(\be)_{j,k\in\bNz^n,|k|<m}$ of complex numbers such that for every $z$ in $a+r\bB_n$,
\begin{equation}\label{eq:Am_into_power_series}
f(z) =
\sum_{j\in\bNz^n}
\sum_{\substack{k\in\bNz^n \\ |k|<m}}
\be_{j,k} (z-a)^j (\conju{z}-\conju{a})^k.
\end{equation}
Moreover, this series converges uniformly on every compact subset of $\bB_n$.
\end{corollary}

\begin{proof}
It is well known that every
holomorphic function on $\bB_n$, decomposes on $\bB_n$ into a power series,
converging on $\bB_n$
and uniformly converging on compact subsets of $\bB_n$.
Applying this fact to each $h_j$ from Corollary~\ref{cor:Am_around_point_a},
we obtain~\eqref{eq:Am_into_power_series}.
\end{proof}

Let us mention a version
of the uniqueness property
for $m$-analytic functions.

\begin{proposition}
\label{prop:polyanalytic_uniqueness}
Let $\Om$ be a connected open set in $\bC^n$,
$\Om_1$ be an open subset of $\Om$,
and $f\in\cA_m(\Om)$
such that $f(z)=0$ for every $z$ in $\Om_1$.
Then $f(z)=0$ for every $z$ in $\Om$.
\end{proposition}

\begin{proof}
For $k$ in $\bN^n$,
the uniqueness property of $k$-analytic functions is proven in~\cite[Section~6.4]{Balk1991}.
The uniqueness property for $m$-analytic functions is a corollary of this fact, since $\cA_m(\Om)\subseteq\cA_{(m,\ldots,m)}(\Om)$.
\end{proof}

To finish this section,
we will show that the class $\cA_m$ with $m$ in $\bN$ is closed under linear changes of variables, while the classes $\cA_k$ with $k\in\bN^n$ are generally not.

\begin{proposition}\label{prop:Am_is_invariant_under_linear_changes_of_variables}
Let $M$ be an invertible $n\times n$ complex matrix
and $f$ in $\cA_m(\Om)$.
Define $g\colon M\Om\to\bC$ by
$g(z)\eqdef f(M^{-1}z)$.
Then $g\in\cA_m(M\Om)$.
\end{proposition}

\begin{proof}
Theorem~\ref{thm:Am_characterization} allows us to work with $\wt{\cA}_m$ instead of $\cA_m$.
Let $f$ be like in Definition~\ref{def:wtAm}.
Then
\[
g(z)=\sum_{|j|<m}h_j(M^{-1} z)\,\conju{(M^{-1} z)}^{\,j}.
\]
The functions $z\mapsto h_j(M^{-1}z)$ are analytic.
Let $M^{-1}=[c_{r,s}]_{r,s=1}^n$.
Then
\[
\conju{(M^{-1}z)}^{\,j}
=\prod_{r=1}^n
\left(\sum_{s=1}^n \conju{c_{r,s}}\,\conju{z_s}\right)^{j_r}.
\]
The last expression is a homogeneous polynomial in $\conju{z_1},\ldots,\conju{z_n}$ of total degree $|j|$,
which is strictly less than $m$
(the same conclusion can also be obtained by Lemma~\ref{lem:f_in_A_p_g_in_A_q_implies_fg_in_A_p+q-1}).
Therefore $g\in\wt{\cA}_m(M\Om)$.
\end{proof}

\begin{proposition}
\label{prop:Ak_is_not_invariant_under_linear_changes_of_variables}
Let $n\ge 2$, $\Om$ be an open subset of $\bC^n$,
$k\in\bN^n$, $k\ne(1,1,\ldots,1)$.
Then there exists a function $f$ in $\cA_k(\Om)$
and an invertible matrix $M$ in $\bC^{n\times n}$
such that the function $g\colon M\Om\to\bC$,
defined by $g(z)\eqdef f(M^{-1}z)$, does not belong to $\cA_k(M\Om)$.
\end{proposition}

\begin{proof}
To simplify the notation, we suppose that $k_1>1$.
The general case is analogous.
Define $M$ in such a manner that
\[
M^{-1}z=(z_1+z_2,z_2-z_1,z_3,\ldots,z_n).
\]
Consider $f\colon\Om\to\bC$,
$f(z)\eqdef \conju{z_1}^{k_1-1}\,\conju{z_2}^{k_2-1}$.
Then
\[
g(z)=(\conju{z_1}+\conju{z_2})^{k_1-1}\,
(\conju{z_2}-\conju{z_1})^{k_2-1}.
\]
In the expansion of the last polynomial, one of the terms is $\conju{z_2}^{k_1+k_2-2}$.
Since $k_1+k_2-2>k_2-1$, we obtain $g\notin\cA_k(M\Om)$, though $f\in\cA_k(\Om)$.
\end{proof}

\section{Weighted mean value property}
\label{sec:weighted_mean_value_property}

In this section we prove that
the value of a $m$-analytic function
at the center of the unit ball $\bB_n$
can be expressed as the integral of this function over the ball, with a certain real radial weight
(Theorem~\ref{thm:weighted_mean_value_property}).
Similar results in the one-dimensional case were proved in~\cite{Koshelev1977,Pessoa2014,HachadiYoussfi2019}.

\subsection*{Jacobi polynomials and their reproducing property}

Some integrals over the unit ball,
written in the spherical coordinates,
reduce to integrals over the unit interval $(0,1)$
with weights of power type at the boundary points $0$ and $1$.
Thereby Jacobi polynomials naturally appear.
They can be defined by Rodrigues formula:
\begin{equation}\label{eq:Jacobi_Rodrigues}
P_m^{(\xi,\eta)}(x)
\eqdef\frac{(-1)^m}{2^m\,m!}\,(1-x)^{-\xi}(1+x)^{-\eta}\;
\frac{\dif^m}{\dif{}x^m}
\Bigl((1-x)^{m+\xi} (1+x)^{m+\eta} \Bigr).
\end{equation}
Here are well-known explicit formulas for $P_m$:
\begin{align}
\label{eq:Jacobi_explicit}
P_m^{(\xi,\eta)}(x)
&=\sum_{s=0}^m \binom{\xi+\eta+m+s}{s}
\binom{\xi+m}{m-s}
\left(\frac{x-1}{2}\right)^s
\\[0.5ex]
\label{eq:Jacobi_explicit2}
&=\sum_{s=0}^m (-1)^s\binom{\xi+\eta+m+s}{s}
\binom{\eta+m}{m-s}
\left(\frac{x+1}{2}\right)^s.
\end{align}
If $\xi,\eta>-1$,
then $(P^{(\xi,\eta)}_m)_{m=0}^\infty$
is an orthogonal family on the interval $(-1,1)$
with respect to the weight $(1-x)^\xi (1+x)^\eta$.
Using~\eqref{eq:Jacobi_Rodrigues}
and integrating by parts yields
the following integral formula:
\begin{align}
\label{eq:Jacobi_with_incomplete_weight}
\int_{-1}^1 P_m^{(\al,\be+1)}(x)\,(1-x)^\al(1+x)^\be\,\dif{}x
=2^{\al+\be+1}(-1)^m\Be(\al+m+1,\be+1). \end{align}

\begin{definition}\label{def:R_polynomial}
Let $m\in\bN$ and $\al,\be>-1$.
We denote by $R_m^{(\al,\be)}$
the following polynomial:
\begin{equation}\label{eq:Rpol_def}
R_m^{(\al,\be)}(t)\eqdef 
\frac{(-1)^m\Be(\al+1,\be+1)}{\Be(\al+m+1,\be+1)}
\,P_m^{(\al,\be+1)}(2t-1).
\end{equation}
\end{definition}

Equivalently, by the symmetry relation for Jacobi polynomials,
\begin{equation}\label{eq:Rpol_def2}
R_m^{(\al,\be)}(t)
=\frac{\Be(\al+1,\be+1)}{\Be(\al+m+1,\be+1)}
\,P_m^{(\be+1,\al)}(1-2t).
\end{equation}
Combining~\eqref{eq:Rpol_def} with~\eqref{eq:Jacobi_explicit2} or~\eqref{eq:Rpol_def2} with~\eqref{eq:Jacobi_explicit},
we get more explicit formulas for $R_m^{(\al,\be)}$:
\begin{align}
\label{eq:R_explicit}
R_m^{(\al,\be)}(t)
&=\frac{\Ga(\al+1)\,\Ga(\be+m+2)}%
{\Ga(\al+\be+2)\,\Ga(\al+m+1)}
\sum_{s=0}^m
\frac{(-1)^s\,\Ga(\al+\be+m+s+2)}{s!\,(m-s)!\,\Ga(\be+s+2)}\,t^s
\\
\label{eq:R_explicit2}
&=\frac{\Ga(\al+1)\,\Ga(\be+m+2)}%
{\Ga(\al+\be+2)\,(\al+m)\,m!}
\sum_{s=0}^m
\frac{(-1)^s\,\binom{m}{s}}{\Be(\al+m,\be+s+2)}\,t^s.
\end{align}
The next simple result was proven in~\cite{BarreraMaximenkoRamos2021}
using the orthogonality of the Jacobi polynomials
and formula~\eqref{eq:Jacobi_with_incomplete_weight}.
Previously,
Hachadi and Youssfi~\cite[formula (5.7)]{HachadiYoussfi2019}
gave another proof for the case $\be=0$.

\begin{proposition}\label{prop:R_repr_on_interval}
Let $m\in\bN$ and $\al,\be>-1$.
Then for every univariate polynomial $h$
with complex coefficients
and $\deg(h)\le m$,
\begin{equation}\label{eq:R_repr_on_interval}
\frac{1}{\Be(\al+1,\be+1)}\int_0^1 h(t) R_m^{(\al,\be)}(t)\,(1-t)^\al t^\be\,\dif{}t=h(0).
\end{equation}
\end{proposition}

The polynomials of degree $\le m$, considered as square-integrable functions on the interval $(0,1)$ with the normalized weight
$\frac{1}{\Be(\al+1,\be+1)}(1-t)^\al t^\be$, form a RKHS.
Formula~\eqref{eq:R_repr_on_interval} means that $R_m^{(\al,\be)}$ is the RK of this space at the point $0$.

As a particular case of~\eqref{eq:R_repr_on_interval},
for every $k$ in $\bNz$ with $k\le m$,
\begin{equation}\label{eq:int_R_power}
\frac{1}{\Be(\al+1,\be+1)}
\int_0^1 R_m^{(\al,\be)}(t)
(1-t)^\al t^{\be+k}\,
\dif{}t
=\de_{k,0}.
\end{equation}

\subsection*{Weighted mean value property of homogeneously polyanalytic functions}

We denote by $\mu$ the Lebesgue measure on $\bC^n$,
by $\bS_n$ the unit sphere in $\bC^n$,
and by $\mu_{\bS_n}$ the (non-normalized) area measure on $\bS_n$.
It is well known~\cite[Section~1.4]{Rudin1980} that
\[
\mu(\bB_n)=\frac{\pi^n}{n!},\qquad
\mu_{\bS_n}(\bS_n) = \frac{2\pi^n}{(n-1)!},
\]
and
\begin{equation}
\label{eq:integral_of_monomials_over_sphere}
\int_{\bS_n} \zeta^j\conju{\zeta}^k
\dif{}\mu_{\bS_n}(\zeta)
=\frac{2\pi^nj!}{(n-1+|j|)!}\cdot\,\de_{j,k}\qquad(j,k\in\bNz^n).
\end{equation}
Given an integrable function $f$ on $\bB_n$,
its integral over $\bB_n$ can be written as
\begin{equation}\label{eq:integral_over_ball}
\int_{\bB_n}f\,\dif{}\mu
=\int_0^1 r^{2n-1}
\left( 
\int_{\bS_n} f(r\zeta)\,\dif{}\mu_{\bS_n}(\zeta)
\right)\dif{}r.
\end{equation}
For $\al>-1$, we denote by $\mu_\al$ the Lebesgue measure on $\bB_n$ with the standard radial weight:
\begin{equation}\label{eq:mu_alpha}
\dif{}\mu_{\al}(z) = c_{\al} (1-|z|^2)^\al \, \dif\mu(z).
\end{equation}
The normalizing constant $c_\al$ is chosen so that $\mu_{\al}(\bB_n)=1$:
\begin{equation}\label{def:constant_weight_c_alpha}
c_\al \eqdef \frac{\Ga(n+\al+1)}{\pi^n\,\Ga(\al+1)}.
\end{equation}

\begin{theorem}\label{thm:weighted_mean_value_property}
Let $f\in\cA_m(\bB_n)$
such that $f\in L^1(\bB_n,\mu_\al)$.
Then
\begin{equation}\label{eq:weighted_mean_value_property}
f(0)=\int_{\bB_n}f(z)
R_{m-1}^{(\al,n-1)}(|z|^2)\,\dif\mu_\al(z).
\end{equation}
\end{theorem}

\begin{proof}
We represent $f$ in the form~\eqref{eq:Am_into_power_series} with $a=0$, then make the change of variables $z=r\zeta$ with $0\le r<1$,
$\zeta\in\bS_n$:
\begin{equation}\label{eq:polyanalytic_series_in_spheric_coordinates}
f(z)
=\sum_{j\in\bNz^n}
\sum_{\substack{k\in\bNz^n \\ |k|<m}}
\be_{j,k} r^{|j|+|k|} \zeta^j\conju{\zeta}^k.
\end{equation}
For every $s$ in $(0,1)$, let $I_s$ be the integral similar to the right-hand side of~\eqref{eq:weighted_mean_value_property},
but over the ball $s\bB_n$:
\[
I_s \eqdef \int_{s\bB_n}f(z)
R_{m-1}^{(\al,n-1)}(|z|^2)\,\dif\mu_\al(z).
\]
Since the series~\eqref{eq:polyanalytic_series_in_spheric_coordinates} converges uniformly over $r$ in $[0,s]$ and $\zeta$ in $\bS_n$,
it can be interchanged with the integral over $s\bB_n$.
Then we apply~\eqref{eq:integral_over_ball}
and~\eqref{eq:integral_of_monomials_over_sphere}:
\begin{align*}
I_s
&=c_\al \sum_{j\in\bNz^n}
\sum_{\substack{k\in\bNz^n \\ |k|<m}}
\be_{j,k}
\int_0^s r^{2n-1+|j|+|k|}
R_{m-1}^{(\al,n-1)}(r^2)(1-r^2)^\al
\left(\int_{\bS_n} \zeta^j\conju{\zeta}^k\,
\dif{}\mu_{\bS_n}(\zeta) \right)\dif{}r
\\
&= c_\al
\sum_{\substack{k\in\bNz^n \\ |k|<m}}
\be_{k,k}\cdot
\frac{\pi^n\,k!}{(n-1+|k|)!}
\int_0^s R_{m-1}^{(\al,n-1)}(t)
(1-t)^\al t^{n-1+|k|}\,\dif{}t.
\end{align*}
The condition $f\in L^1(\bB_n,\mu_\al)$
implies that $I_s\to I_1$, as $s\to 1$.
Passing to this limit and using~\eqref{eq:int_R_power},
we finally obtain
\begin{align*}
I_1
&=\frac{\Ga(n+\al+1)}{\Ga(\al+1)}
\sum_{\substack{k\in\bNz^n \\ |k|<m}}
\be_{k,k}\cdot
\frac{k!}{(n-1+|k|)!}
\int_0^1 R_{m-1}^{(\al,n-1)}(t)
(1-t)^\al t^{n-1+|k|}\,\dif{}t
\\
&=\frac{\Ga(\al+n+1)}{\Ga(\al+1)} \sum_{|k|<m} \be_{k,k} \frac{k!}{(n-1+|k|)!} \cdot 
\de_{k,0}\,\Be(\al+1,n)
=\be_{0,0}=f(0).
\qedhere
\end{align*}
\end{proof}

Here is an analog of~\eqref{eq:weighted_mean_value_property}
for an arbitrary ball and for $\al=0$.
\begin{corollary}\label{cor:mean_value_property}
Let $\Om$ be an open subset of $\bC^n$,
$f\in\cA_m(\Om)$,
$a\in\Om$,
and $r>0$ such that $a+r\bB_n\subseteq\Om$.
Suppose that $f\in L^1(a+r\bB_n,\mu)$.
Then
\begin{equation}\label{eq:mean_value_property}
f(a)=\frac{n!}{\pi^n}\,\frac{1}{r^{2n}}\int_{a+r\bB_n}f(z)
R_{m-1}^{(0,n-1)}\left(\frac{|z-a|^2}{r^2}\right)\,\dif\mu(z).
\end{equation}
\end{corollary}

\subsection*{Bergman spaces of homogeneously polyanalytic functions}

In the rest of this section, we suppose that $\Om$, $m$, $W$, $\nu$ are like in Definition~\ref{def:Bergman_polyanalytic}.
Using~\eqref{eq:mean_value_property},
it is easy to prove the upcoming Lemma~\ref{lem:evaluation_functionals_on_Am2_are_bounded} and Proposition~\ref{prop:Am2_is_RKHS}.
See similar proofs for the one-dimensional case in~\cite[Lemma~4.3, Proposition~4.4]%
{BarreraMaximenkoRamos2021}.

\begin{lemma}\label{lem:evaluation_functionals_on_Am2_are_bounded}
Let $K$ be a compact subset of $\Om$.
Then there exists a number $C_{m,W,K}>0$
such that for every $f$ in $\cA_m^2(\Om,\nu)$
and every $z$ in $K$,
\begin{equation}\label{eq:evaluation_functionals_on_Am_are_bounded}
|f(z)|\le C_{m,W,K} \|f\|_{\cA_m^2(\Om,\nu)}.
\end{equation}
\end{lemma}

\begin{proposition}\label{prop:Am2_is_RKHS}
$\cA_m^2(\Om,\nu)$ is a RKHS.
\end{proposition}

As a corollary, the spaces $\cA_{(m)}^2(\Om,\nu)$ are also RKHS.

\begin{proposition}\label{prop:L2_is_sum_true_poly}
In the conditions of Definition~\ref{def:Bergman_polyanalytic},
suppose additionally that $\Om$ is bounded
and $\nu$ is finite.
Then
\begin{equation}\label{eq:L2_as_direct_sum}
L^2(\Om,\nu)
=\bigoplus_{m=1}^\infty\cA_{(m)}^2(\Om,\nu).
\end{equation}
\end{proposition}

\begin{proof}
This is a simple consequence of three facts:
1) the continuous functions with compact supports form a dense subset of $L^2(\Om,\nu)$;
2) by the Stone--Weierstrass theorem,
every continuous function on the closure of $\Om$ can be uniformly approximated by polynomials in $z_1,\ldots,z_n,\conju{z_1},\ldots,\conju{z_n}$; 
and 3) the norm of $L^2(\Om,\nu)$ can be estimated from above by a constant multiple of the maximum-norm.
\end{proof}

In the one-dimensional case,
the ``true-$m$-analytic'' spaces
$\cA_{(m)}^2$ were studied by Ramazanov~\cite{Ramazanov1999} and Vasilevski~\cite{Vasilevski2000polyFock,Vasilevski2008book}.
According to~\cite{Vasilevski2000polyFock}
(see also another proof in~\cite{MaximenkoTelleria2020}),
the decomposition~\eqref{eq:L2_as_direct_sum}
holds for the poly-Fock space
$\cA_m^2(\bC,\enumber^{-|z|^2}\,\dif\mu)$.
On the other hand,
if $\Om$ is the upper halfplane $\bH_1$ with the Lebesgue measure,
then $L^2(\bH_1)$ decomposes into the orthogonal sum of the spaces $\cA_{(m)}^2(\bH_1)$ and their conjugates~\cite[Theorem~3.3.5]{Vasilevski2008book},
and~\eqref{eq:L2_as_direct_sum} fails.
It is natural to ask if Proposition~\ref{prop:L2_is_sum_true_poly}
remains true if $\nu(\Om)<+\infty$,
without assuming $\Om$ to be bounded.

\section{Pushforward reproducing kernel}
\label{sec:pushforward_RK}

In this section we show how to transform a RK using a weighted change of variables.
First, we deal with abstract positive kernels~\cite{Aronszajn1950},
then we consider reproducing kernels
in Hilbert spaces.

Let $X$ be a non-empty set.
We denote by $\bC^X$ the complex vector space of all functions $X\to\bC$ with pointwise operations.
A family $(K_x)_{x\in X}$ with values in $\bC^X$
is called a \emph{positive kernel} on $X$
if for every $m$ in $\bN$,
every $x_1,\ldots,x_m$ in $X$
and every $\al_1,\ldots,\al_m$ in $\bC$,
\[
\sum_{r,s=1}^m \al_r \overline{\al_s} K_{x_r}(x_s)\ge0.
\]

\begin{proposition}\label{prop:pusforward_kernel}
Let $X,Y$ be non-empty sets,
$\psi\colon Y\to X$
and $J\colon Y\to\bC$
be some functions,
and $(K_x)_{x\in X}$
be a positive kernel on $X$.
Then the family $(L_u)_{u\in Y}$,
defined by
\[
L_u(v)\eqdef
\overline{J(u)}\,J(v) K_{\psi(u)}(\psi(v)),
\]
is a positive kernel on $Y$.
\end{proposition}

\begin{proof}
Let $m\in\bN$, $u_1,\ldots,u_m\in Y$, $\al_1,\ldots,\al_m\in\bC$.
For every $s$ in $\{1,\ldots,m\}$ put
$x_s\eqdef\psi(u_s)$
and
$\be_s\eqdef\overline{J(u_s)}\,\al_s$.
Then
\[
\sum_{r,s=1}^m \al_r \overline{\al_s} L_{u_r}(u_s)
=\sum_{r,s=1}^m \be_r \overline{\be_s} K_{x_r}(x_s)
\ge0.
\qedhere
\]
\end{proof}

Let $X$ be a non-empty set.
We say that $H$ is a Hilbert space of functions on $X$ if $H$ is a vector subspace of $\bC^X$, provided with an inner product and complete with respect to the corresponding norm.
Furthermore, if $x\in X$, $\cK\in H$ and $\langle f,\cK\rangle = f(x)$ for every $f$ in $H$, then we say that $\cK$ is a reproducing kernel of $H$ at the point $x$.
In case of existence, this function is unique.

\begin{proposition}
\label{prop:pushforward_RK_one_point}
Let $X,Y$ be non-empty sets,
$\psi\colon Y\to X$
and $J\colon Y\to\bC$
be some functions,
$H_1$ be a Hilbert space of functions over $X$,
$H_2$ be a Hilbert space of functions over $Y$, and
\[
(Uf)(z)\eqdef J(z)f(\psi(z))
\]
be a well-defined unitary operator
mapping $H_1$ onto $H_2$.
Suppose that $u\in Y$
and $\cK$ be the reproducing kernel of $H_1$ at the point $\psi(u)$.
Then the function $\cL\colon Y\to\bC$, defined by the following rule,
is the reproducing kernel of $H_2$
at the point $u$:
\[
\cL(v)
\eqdef\overline{J(u)}\,J(v)\,\cK(\psi(v)).
\]
\end{proposition}

\begin{proof}
Let $g\in H_2$ and $f=U^{-1}g$.
Then
\[
g(u)= J(u)f(\psi(u))
= J(u)
\langle f,\cK\rangle_{H_1}
= \langle g, \overline{J(u)}\,U\cK\rangle_{H_2}.
\]
Defining $\cL$ by
$\cL(v)
=\overline{J(u)}(U\cK)(v)
=\overline{J(u)}J(v)\cK(\psi(v))$,
we get the RK of $H_2$ at $u$.
\end{proof}

\begin{proposition}\label{prop:pushforward_RK}
Let $X,Y$ be non-empty sets,
$\psi\colon Y\to X$
and $J\colon Y\to\bC$
be some functions,
$H_1$ be a Hilbert space of functions over $X$
with reproducing kernel $(K_x)_{x\in X}$,
$H_2$ be a Hilbert space of functions over $Y$, and
\[
(Uf)(z)\eqdef J(z)f(\psi(z))
\]
be a well-defined unitary operator
mapping $H_1$ onto $H_2$.
Then $H_2$ is a RKHS,
and its reproducing kernel
$(L_u)_{u\in Y}$ is given by
\[
L_u(v)=\overline{J(u)}\,J(v)\,
K_{\psi(u)}(\psi(v)).
\]
\end{proposition}

\begin{proof}
Apply Proposition~\ref{prop:pushforward_RK_one_point} at every point $u$ of $Y$.
\end{proof}

As a simple application of the this scheme, let us express the Berezin transform in $H_2$ via the Berezin transform in $H_1$.
Given a Hilbert space $H$, we denote by $\cB(H)$
the C*-algebra of all bounded linear operators acting in $H$.
Given a set $X$, we denote by $B(X)$ the Banach space of all bounded functions on $X$, with the supremum norm.
If $H$ is a RKHS over $X$ and its RK satisfies $\|K_x\|_H\ne0$ for every $x$ in $X$, then the
\emph{Berezin transform}
$\Ber_H\colon\cB(H)\to B(X)$
is defined by
\[
\Ber_H(A)(x)
\eqdef\frac{\langle AK_x,K_x\rangle_H}%
{\langle K_x,K_x\rangle_H}
\qquad(A\in\cB(H),\ x\in X).
\]

\begin{proposition}
\label{prop:Berezin_H2_in_terms_of_Berezin_H1}
In the conditions of Proposition~\ref{prop:pushforward_RK},
suppose that $\|K_x\|_{H_1}\ne0$
for every $x$ in $X$
and $J(u)\ne0$ for every $u$ in $Y$.
Then
\[
\Ber_{H_2}(A)(u)
=\Ber_{H_1}(U^\ast A U)(\psi(u))\qquad
(A\in\cB(H_2),\ u\in Y).
\]
\end{proposition}

\begin{proof}
As we have seen in Proposition~\ref{prop:pushforward_RK_one_point}, $L_u(v)=\conju{J(u)}(UK_{\psi(u)})(v)$. Therefore,
\begin{align*}
\Ber_{H_2}(A)(u) 
&=\frac{\langle AL_u,L_u
\rangle_{H_2}}{\|L_u\|^2}
= \frac{|J(u)|^2\langle AUK_{\psi(u)}, UK_{\psi(u)}\rangle_{H_2}}%
{|J(u)|^2\|UK_{\psi(u)}\|^2}\\
&= \frac{\langle U^* AUK_{\psi(u)},K_{\psi(u)}\rangle_{H_1}}%
{\|K_{\psi(u)}\|^2} 
= \Ber_{H_1}(U^*AU)(\psi(u)). \qedhere
\end{align*}
\end{proof}

\begin{corollary}\label{cor:Berezin_H1_injective_implies_Berezin_H2_injective}
In the conditions of Proposition~\ref{prop:Berezin_H2_in_terms_of_Berezin_H1},
suppose that $\Ber_{H_1}$ is injective.
Then $\Ber_{H_2}$ is also injective.
Moreover, if $\psi$ is a bijection,
than the injectivity of $\Ber_{H_1}$
is equivalent to the injectivity of $\Ber_{H_2}$.
\end{corollary}

\section{Reproducing kernel on the unit ball}
\label{sec:RK_unit_ball}

In this section we consider the domain $\Om=\bB_n$ with
the standard radial measure $\mu_\al$,
given by~\eqref{eq:mu_alpha}.
Using the weighted mean value property
and appropriate unitary operators,
we compute the RK of $\cA_m^2(\bB_n,\mu_\al)$.

\subsection*{On the unit ball biholomorphisms}

For a fixed $a$ in $\bB_n\setminus\{0\}$,
we denote by $\phi_a$ the function $\bB_n\to\bB_n$, defined by
\begin{equation}\label{eq:phi_def}
\phi_a(z) \eqdef
\frac{a
-\frac{\langle z,a\rangle}{\langle a,a\rangle}\,a
-\sqrt{1-|a|^2}\,\left(z
-\frac{\langle z,a\rangle}%
{\langle a,a\rangle}a\right)}%
{1-\langle z,a \rangle}.
\end{equation}
For $a=0$, $\phi_a(z)\eqdef z$.
It is well
known~\cite[Theorem 2.2.2]{Rudin1980}
that for every $a$ in $\bB_n$,
$\phi_a$ is a biholomorphism of $\bB_n$,
$\phi_a(\phi_a(z))=z$
for every $z$ in $\bB_n$,
$\phi_a(0)=a$, $\phi_a(a)=0$, and
\begin{equation}
\label{eq:phi_property_zw}
1-\langle \phi_a(z),\phi_a(w)\rangle
=\frac{(1-\langle a,a\rangle)(1-\langle z,w\rangle)}%
{(1-\langle z,a\rangle)(1-\langle a,w\rangle)}.
\end{equation}
Here are particular cases of~\eqref{eq:phi_property_zw},
with $w=z$ and $w=0$, respectively:
\begin{align}
\label{eq:phi_property_zz}
1-|\phi_a(z)|^2
&= \frac{(1-|a|^2)(1-|z|^2)}%
{|1-\langle z,a\rangle|^2},
\\[0.5ex]
\label{eq:phi_property_z0}
1-\langle \phi_a(z),a\rangle
&=\frac{1-|a|^2}{1-\langle z,a\rangle}.
\end{align}
The real Jacobian of $\phi_a$ is~\cite[Lemma~1.7]{Zhu2005}
\begin{equation}\label{eq:phi_real_Jacobian}
(J_\bR \phi_a)(z)
= \left(
\frac{1-|a|^2}{|1-\langle z,a \rangle|^2}
\right)^{n+1}.
\end{equation}
We denote by $\rho_{\bB_n}(z,w)$
the expression $|\phi_z(w)|$,
known as the \emph{pseudohyperbolic distance} between $z$ and $w$,
see~\cite[Corollary~1.22]{Zhu2005}
or~\cite{DurenWeir2007}.
Formula~\eqref{eq:phi_property_zz}
provides a simple recipe to compute $\rho_{\bB_n}(z,w)$.

\subsection*{A factor to preserve the polyanalyticity}

\begin{definition}\label{def:p}
Given $a$ in $\bB_n$, we define
$p_{m,a}\colon\bB_n\to\bC$ by
\[
p_{m,a}(z)
\eqdef
\left( \frac{1-\langle a,z \rangle}{1-\langle z,a \rangle} \right)^{m-1}.
\]
\end{definition}

In the one-dimensional case,
the function $p_{m,a}$ was introduced and studied by Pessoa~\cite{Pessoa2014}.
As it is shown in the proof of Lemma~\ref{lem:composition_by_Pessoa},
the main purpose of $p_{m,a}$
is to eliminate the denominators in the multi-powers
of $\overline{\phi_a(z)}$.

\begin{lemma}
For every $a,z$ in $\bB_n$,
\begin{gather}
\label{eq:p_absolute_value}
|p_{m,a}(z)|=1,
\\[0.5ex]
\label{eq:p_inversion}
p_{m,a}(\phi_a(z))p_{m,a}(z)=1.
\end{gather}
\end{lemma}

\begin{proof}
Formula~\eqref{eq:p_absolute_value} follows directly from the definition of $p_{m,a}$.
Identity~\eqref{eq:p_inversion}
is easy to verify using~\eqref{eq:phi_property_z0}.
\end{proof}

\begin{lemma}
\label{lem:composition_by_Pessoa}
Let $a\in\bB_n$ and $f\in\cA_m(\bB_n)$.
Then 
$(f\circ\phi_a)\cdot p_{m,a}
\in\cA_m(\bB_n)$.
\end{lemma}

\begin{proof}
Let $f$ be of the form~\eqref{eq:polyanalytic_function_as_polynomial}.
Denote by $N_a(z)$ the numerator of~\eqref{eq:phi_def};
it is a polynomial of degree $1$ in $z_1,\ldots,z_n$.
Then,
\begin{align*}
f(\phi_a(z))p_{m,\al}(z)
&=\left(\frac{1-\langle a,z \rangle}%
{1-\langle z,a \rangle} \right)^{m-1}
\sum_{|j|<m}h_j(\phi_a(z))
\,\frac{\overline{N_a(z)}^{\,j}}%
{(1-\langle a,z\rangle)^{|j|}}
\\
&=\sum_{|j|<m}
\frac{h_j(\phi_a(z))}%
{(1-\langle z,a\rangle)^{m-1}}\,
\overline{N_a(z)}^{\,j}\,
(1-\langle a,z\rangle)^{m-1-|j|}.
\end{align*}
The quotients in the last sum are analytic functions of $z$.
The multi-power $\overline{N_a(z)}^{\,j}$
is a polynomial in $\conju{z_1},\ldots,\conju{z_n}$
of total degree $|j|$,
and the expression
$(1-\langle a,z\rangle)^{m-1-|j|}$
is a polynomial in
$\conju{z_1},\ldots,\conju{z_n}$
of total degree $m-1-|j|$.
Therefore, the whole sum
is a polynomial in $\conju{z_1},\ldots,\conju{z_m}$ of total degree at most $m-1$, with some analytic coefficients.
\end{proof}

\subsection*{A factor to preserve the norm}

\begin{remark}\label{rem:def_power}
In the upcoming formula for $g_{\al,a}$ and in some other formulas of this paper,
we work with (non necesarily integer) powers of complex numbers.
Given $t$ in $\bC\setminus\{0\}$
and $\be$ in $\bC$,
we define $t^\be$ as $\exp(\be\log(t))$,
where
$\log(t)=\log_\bR|t|+\iu\arg(t)$,
$\log_\bR|t|$ is the real logarithm of $|t|$, and $\arg(t)$ is the principal argument of $t$,
belonging to $(-\pi,\pi]$.
\end{remark}

Given $a$ in $\bB_n$, we denote by $g_{\al,a}$ the following function $\bB_n\to\bC$:
\begin{equation}\label{eq:gdef}
g_{\al,a}(z)
\eqdef
\frac{(1-|a|^2)^{\frac{n+1+\al}{2}}}{(1-\langle z,a \rangle)^{n+1+\al}}.
\end{equation}
This function and their properties stated below appear in Vukoti\'{c}~\cite{Vukotic1993}.
See also~\cite[Proposition~1.13]{Zhu2005}
or~\cite[formula~(2.4)]{Englis2007}.
By~\eqref{eq:phi_property_z0},
\begin{equation}
\label{eq:g_product}
g_{\al,a}(\phi_a(z))g_{\al,a}(z)=1.
\end{equation}
By~\eqref{eq:g_product},
\eqref{eq:phi_property_zz},
and~\eqref{eq:phi_real_Jacobian},
\begin{equation}
\label{eq:g_main_property}
|g_{\al,a}(\phi_a(w))|^2 (J_\bR \phi_a)(w)
(1-|\phi_a(w)|^2)^\al
=(1-|w|^2)^\al.
\end{equation}
Using \eqref{eq:g_main_property}
and the change of variables $w=\phi_a(z)$,
one easily shows that for every $f$ in $f\in L^2(\bB_n,\mu_\al)$,
\begin{equation}\label{eq:g_isometry}
\|(f\circ\phi_a)\cdot g_{\al,a}\|_{L^2(\bB_n,\mu_\al)}
=\|f\|_{L^2(\bB_n,\mu_\al)}.
\end{equation}

\subsection*{A weighted shift operator
preserving $\cA_m^2(\bB_n,\mu_\al)$}

\begin{definition}\label{def:U}
Given $a$ in $\bB_n$, we define
$U_a \colon \cA_m^2(\bB_n,\mu_\al) \to \cA_m^2(\bB_n,\mu_\al) $ by
\[
(U_a f)(z) \eqdef 
f(\phi_a(z)) p_{m,a}(z) g_{\al,a}(z).
\]
\end{definition}

\begin{proposition}\label{prop:Ua_is_unitary}
Let $a\in\bB^n$.
Then $U_a$ is a unitary operator in $\cA_m^2(\bB_n,\mu_\al)$,
and $U_a^2=I$.
\end{proposition}

\begin{proof}
Given $f$ in $\cA_m^2(\bB_n,\mu_\al)$,
Lemma~\ref{lem:composition_by_Pessoa}
assures that $U_a f\in\cA_m(\bB_n)$.
Formula~\eqref{eq:g_isometry},
combined with~\eqref{eq:p_absolute_value},
implies that $U_a$ is an isometry.
Finally, \eqref{eq:p_inversion}
and~\eqref{eq:g_product}
yield the involutive property $U_a^2=I$.
\end{proof}

\subsection*{Computation of the RK on the unit ball}

Recall that $R_m^{(\al,\be)}$ is defined by~\eqref{eq:Rpol_def}
and $\rho_{\bB_n}(z,w)$ denotes $|\phi_z(w)|$.

\begin{theorem}\label{thm:RK_ball}
Let $n,m\in\bN$ and $\al>-1$.
Then for every $z$ in $\bB_n$,
the following function $K_z$
is the reproducing kernel of $\cA_m^2(\bB_n,\mu_\al)$
at the point $z$:
\begin{equation}
\label{eq:RK_ball}
K_z(w)
=\frac{(1-\langle z,w\rangle)^{m-1}}%
{(1-\langle w,z\rangle)^{n+m+\al}}\,
R_{m-1}^{(\al,n-1)}(\rho_{\bB_n}(z,w)^2).
\end{equation}
\end{theorem}

\begin{proof}
For $z=0$, the function defined by the right-hand side of~\eqref{eq:RK_ball} simplifies to
\[
K_0(w)=R_{m-1}^{(\al,n-1)}(|w|^2).
\]
Theorem~\ref{thm:weighted_mean_value_property} means that $K_0$ is indeed the RK at the point $0$.
Now, for $z$ in $\bB_n$,
we apply Proposition~\ref{prop:pushforward_RK_one_point}
with $H_1=H_2=\cA_m^2(\bB_n)$,
$\phi_z$ instead of $\psi$,
and $J_z\eqdef p_{m,z}g_{\al,z}$.
Since $\phi_z(z)=0$, we obtain
\begin{equation}
\label{eq:reproducing_property_at_any_z}
K_z(w)
= \conju{J_z(z)} J_z(w) K_0(\phi_z(w)).
\end{equation}
It is easy to see that $J_z(z)=(1-|z|^2)^{-\frac{n+1+\al}{2}}$.
So, after some simplifications,
we arrive at~\eqref{eq:RK_ball}:
\begin{align*}
K_z(w) &= \frac{1}{(1-|z|^2)^{\frac{n+1+\al}{2}}} 
\left( \frac{1-\langle z,w \rangle}{1-\langle w,z \rangle} \right)^{m-1} 
\frac{(1-|z|^2)^{\frac{n+1+\al}{2}}}{(1-\langle w,z \rangle)^{n+1+\al}}
R_{m-1}^{(\al,n-1)} (|\phi_z(w)|^2) \\
&= \frac{(1-\langle z,w\rangle)^{m-1}}%
{(1-\langle w,z\rangle)^{n+m+\al}}\;
R_{m-1}^{(\al,n-1)}(\rho_{\bB_n}(z,w)^2).
\qedhere
\end{align*}
\end{proof}

Formula~\eqref{eq:RK_ball} is a natural generalization of previous results:
 \cite{Koshelev1977,Pessoa2013} for $n=1$ and $\al=0$, \cite{HachadiYoussfi2019} for $n=1$ and $\al>-1$, and \cite[Theorem 2.7]{Zhu2005} for $m=1$.

\begin{corollary}
\label{cor:RK_ball_norm}
Let $n,m\in\bN$ and $\al>-1$.
Then for every $z$ in $\bB_n$,
\begin{equation}\label{eq:Kz_ball_norm}
\|K_z\|_{\cA_m^2(\bB_n,\mu_\al)}^2
=K_z(z)
=\binom{n+m-1}{n}\,%
\frac{\Be(\al+1,n)}%
{\Be(\al+m,n)}\,
\frac{1}{(1-|z|^2)^{n+\al+1}}.
\end{equation}
\end{corollary}

\begin{remark}
We get other formulas, equivalent to~\eqref{eq:RK_ball}, using~\eqref{eq:phi_property_zz} and~\eqref{eq:Jacobi_explicit}:
\begin{align}
\label{eq:RK_ball_Jacobi}
K_z(w)
&= \frac{(1-\langle z,w\rangle)^{m-1}}%
{(1-\langle w,z\rangle)^{n+m+\al}}\,%
\frac{(-1)^{m-1}\Be(\al+1,n)}{\Be(\al+m,n)}
P_{m-1}^{(\al,n)}(2\rho_{\bB_n}(z,w)^2-1)
\\[1.5ex]
\label{eq:RK_ball_Jacobi_2}
&=\frac{(1-\langle z,w\rangle)^{m-1}}%
{(1-\langle w,z\rangle)^{n+m+\al}}
\frac{(-1)^{m-1}\Be(\al+1,n)}{\Be(\al+m,n)}\;%
P_{m-1}^{(\al,n)}
\left(1-\frac{2(1-|z|^2)(1-|w|^2)}{|1-\langle w,z\rangle|^2}\right)
\\
\nonumber
&=
\frac{(1-\langle z,w\rangle)^{m-1}}%
{(1-\langle w,z\rangle)^{n+m+\al}}\,%
\frac{(-1)^{m-1}\,\Gamma(\al+1)}{\Gamma(\al+n+1)\,(m-1)!} \times
\\
\label{eq:RK_ball_explicit}
&\qquad \times\sum_{s=0}^{m-1}%
(-1)^s \binom{m-1}{s} \frac{\Gamma(\al+m+n+s)}{\Gamma(\al+s+1)}
\left(\frac{(1-|z|^2)(1-|w|^2)}{|1-\langle w,z\rangle|^2}\right)^s.
\end{align}
\end{remark}

\begin{remark}\label{rem:rotations_in_the_ball}
If $M$ is a unitary $n$ by $n$ matrix,
then the RK computed in Theorem~\ref{thm:RK_ball} is invariant under the simultaneous action of $M$ in both arguments:
\[
K_{Mz}(Mw)=K_z(w)\qquad(z,w\in\bB^n).
\]
Therefore, by~\cite[Proposition~4.1]{MaximenkoTelleria2020}, the space $\cA_m^2(\bB^n,\mu_\al)$
is invariant under the action of the rotation operator
\[
(R_M f)(z)\eqdef f(M^{-1}z).
\]
This follows also directly from~Proposition~\ref{prop:Am_is_invariant_under_linear_changes_of_variables}.
Notice that the unitary matrices include permutation matrices, diagonal matrices with unimodular complex entries, and real rotations in any two coordinates.
\end{remark}

\begin{remark}\label{rem:biholomorphic_changes_of_variables_in_the_ball}
Generalizing ideas of this section, it is possible to construct a unitary weighted shift operator $U_\phi$ acting in $\cA_m^2(\bB_n,\mu_\al)$,
for every biholomorphism $\phi$ of $\bB_n$.
\end{remark}

The next result was published by Engli\v{s}~ \cite[Section 2]{Englis2007}
for RKHS of harmonic functions.
We reformulate it for our situation and recall the idea of the proof.

\begin{proposition}\label{prop:Berezin_ball}
Let $H=\cA_m^2(\bB_n,\mu_\al)$, with $n\ge 1$ and $m\ge 2$. Then $\Ber_H$ is not injective.
\end{proposition}

\begin{proof}
The functions $f(z)=z_1$ and $g(z)=\conju{z_1}$ are linearly independent elements of $H$.
Therefore,
the operator
$Sh
=\langle h,f\rangle_H f
-\langle h,g\rangle_H g$ is not zero, but the Berezin transform maps it into the zero function.
\end{proof}

\clearpage
\section{Reproducing kernel on the Siegel domain}
\label{sec:RK_Siegel}

Let $n,m\in\bN$ and $\al>-1$.
In this section we compute the RK
of the space $\cA_m^2(\bH_n,\nu_\al)$,
where $\bH_n$ is the standard Siegel domain (which can be considered as an unbounded realization of the unit ball) and $\nu_\al$ is a usual weighted measure on $\bH_n$:
\begin{gather}
\label{Siegel}
\bH_n\eqdef \{\xi=(\xi',\xi_n)\in\bC^{n-1}\times\bC\colon
\Im(\xi_n)-|\xi'|^2>0\},
\\[0.5ex]
\label{eq:measure_on_Siegel}
\dif\nu_\al(\xi)
\eqdef\frac{c_\al}{4}
(\Im(\xi_n)-|\xi'|^2)^\al\,\dif{}\mu(\xi).
\end{gather}
For this purpose,
we will construct a unitary operator $V\colon\cA_m^2(\bB_n,\mu_\al)\to\cA_m^2(\bH_n,\nu_\al)$,
using some recipes from~\cite[Section~2]{QuirogaVasilevski2007} and an analog of the Pessoa factor which helps to preserve the polyanalyticity.

\subsection*{Cayley transform}

Following~\cite[Section~2]{QuirogaVasilevski2007}, we employ the biholomorphism $\omega\colon\bB_n\to\bH_n$ defined by
\[
\omega(z)
\eqdef \left(\iu\,\frac{z_1}{1+z_n},\ldots,
\iu\,\frac{z_{n-1}}{1+z_n},
\iu\,\frac{1-z_n}{1+z_n}\right).
\]
Its inverse $\psi\colon\bH_n\to\bB_n$ is given by
\[
\psi(\xi)\eqdef
\left(-\frac{2\iu \xi_1}{1-\iu \xi_n},\ldots,
-\frac{2\iu \xi_{n-1}}{1-\iu \xi_n},
\frac{1+\iu \xi_n}{1-\iu \xi_n}\right).
\]
By a direct computation,
\begin{equation}\label{eq:psi_inner_prod}
1-\langle\psi(\xi),\psi(\eta)\rangle
= 4\,\frac{\frac{\xi_n-\conju{\eta_n}}{2\iu}-\langle \xi',\eta'\rangle}%
{(1-i\xi_n)(1+i\conju{\eta_n})}.
\end{equation}
In particular,
\begin{equation}\label{eq:psi_abs}
1 - |\psi(\xi)|^2
= 4\,\frac{\Im(\xi_n) -|\xi'|^2}{|1-i\xi_n|^2}.
\end{equation}
The complex Jacobian matrices of $\psi$ and $\omega$ are triangular, and their determinants are easy to compute:
\begin{equation}
\label{eq:psi_complex_Jacobian}
(J_\bC\omega)(z)
=-\frac{2\iu^n}%
{(1+z_n)^{n+1}},
\qquad
(J_\bC \psi)(\xi)
=-\frac{(-2\iu)^n}%
{(1-\iu \xi_n)^{n+1}}.
\end{equation}
Therefore, the real Jacobians of $\omega$ and $\psi$ are
\begin{equation}
\label{eq:psi_real_Jacobian}
(J_\bR\omega)(z)
=
\frac{4}{|1+z_n|^{2(n+1)}},\qquad
(J_\bR \psi)(\xi)
=\frac{4^n}{|1-\iu\xi_n|^{2(n+1)}}.
\end{equation}

\subsection*{Pseudohyperbolic distance on the Siegel domain}

\begin{definition}
\label{def:pseudohyperbolic_distance_Siegel}
Define a distance on $\bH_n$ by
\begin{equation}\label{eq:pseudohyperbolic_distance_Siegel}
\rho_{\bH_n}(\xi,\eta)\eqdef\rho_{\bB_n}(\psi(\xi),\psi(\eta)).
\end{equation}
\end{definition}

The following proposition provides an efficient formula to compute $\rho_{\bH_n}(\xi,\eta)$.

\begin{proposition}
For every $\xi,\eta$ in $\bH_n$,
\begin{equation}
\label{eq:rho_Siegel_efficient}
1-\rho_{\bH_n}(\xi,\eta)^2
=\frac{(\Im(\xi_n)-|\xi'|^2)%
(\Im(\eta_n)-|\eta'|^2)}%
{\left|\frac{\xi_n-\conju{\eta_n}}{2\iu}%
-\langle \xi',\eta'\rangle\right|^2}.
\end{equation}
\end{proposition}

\begin{proof}
Substitute $\psi(\xi)$ and $\psi(\eta)$ instead of $z$ and $w$ in~\eqref{eq:phi_property_zz}:
\[
1-\rho_{\bH_n}(\xi,\eta)^2
=1-\rho_{\bB_n}(\psi(\xi),\psi(\eta))^2
= \frac{(1-|\psi(\xi)|^2)%
(1-|\psi(\eta)|^2)}%
{|1
-\langle \psi(\xi),\psi(\eta)\rangle|^2}.
\]
Applying~\eqref{eq:psi_inner_prod}
and~\eqref{eq:psi_abs}
we obtain~\eqref{eq:rho_Siegel_efficient}.
\end{proof}

\begin{remark}\label{rem:pseudohyperbolic_distance_upper_halfplane}
For $n=1$, formulas~\eqref{eq:pseudohyperbolic_distance_Siegel} and~\eqref{eq:rho_Siegel_efficient} simplify to
\begin{equation}\label{eq:pseudohyperbolic_distance_upper_halfplane}
\rho_{\bH_1}(\xi,\eta)
=\frac{|\xi-\eta|}{|\conju{\xi}-\eta|},\qquad
1-\rho_{\bH_1}(\xi,\eta)^2
=\frac{4\Im(\xi)\Im(\eta)}{|\conju{\xi}-\eta|^2}.
\end{equation}
\end{remark}

\subsection*{A factor to preserve the norm when passing from $\bH_n$ to $\bB_n$}

The material of this subsection
is equivalent to some computations from
\cite[Section~2]{QuirogaVasilevski2007}.
Define $h_\al\colon\bH_n\to\bC$ by
\begin{equation}\label{eq:hdef}
h_\al(\xi)
\eqdef\left(\frac{2}{1-\iu \xi_n}\right)^{n+\al+1}.
\end{equation}

\begin{lemma}\label{lem:h_main_property}
For every $\xi$ in $\bH_n$,
\begin{equation}\label{eq:h_abs}
|h_\al(\xi)|^2
=\frac{4(1-|\psi(\xi)|^2)^\al
(J_\bR \psi)(\xi)}%
{(\Im(\xi_n)-|\xi'|^2)^\al}.
\end{equation}
For any $z$ in $\bB_n$,
\begin{equation}\label{eq:h_main_property}
\frac{1}{4}
|h_{\al}(\omega(z))|^2 
\,\left(\frac{1-|z|^2}{|1+z_n|^2}\right)^\al
\,(J_\bR \omega)(z)
=(1-|z|^2)^\al.
\end{equation}
\end{lemma}

\begin{proof}
Formula~\eqref{eq:h_abs} is obtained by~\eqref{eq:psi_abs} and~\eqref{eq:psi_real_Jacobian}.
Then~\eqref{eq:h_main_property} follows from~\eqref{eq:h_abs} and the well-known formula for the Jacobian of the inverse function.
\end{proof}

\begin{lemma}\label{lem:h_isometry}
Let $u\in L^2(\bB_n,\mu_\al)$.
Then
\begin{equation}\label{eq:h_isometry}
\|(u\circ\psi)\cdot h_\al\|_{L^2(\bH_n,\nu_\al)}
=\|u\|_{L^2(\bB_n,\mu_\al)}.
\end{equation}
\end{lemma}

\begin{proof}
First, using \eqref{eq:psi_abs}, we observe that the change of variable $z=\psi(\zeta)$ transforms the weight function in the following way:
\[
(\Im(\zeta_n)-|\zeta'|^2)^\al
= \left(\frac{1-|z|^2}{|1+z_n|^2}\right)^\al.
\]
Apply this change of variables in the integral:
\begin{align*}
\|(u\circ\psi)\cdot h_{\al}\|_{L^2(\bH_n,\nu_\al)}^2
&=
\frac{c_\al}{4}
\int_{\bH^n}|u(\psi(\zeta))|^2
|h_{\al}(\zeta)|^2 (\Im(\zeta_n)-|\zeta'|^2)^\al
\dif\mu(\zeta)
\\
&=
c_\al
\int_{\bB^n}|u(z)|^2
|h_{\al}(\omega(z))|^2 
\,\left(\frac{1-|z|^2}{|1+z_n|^2}\right)^\al
\,(J_\bR \omega)(z)\,\dif\mu(z)
\\
&=
c_\al
\int_{\bB^n}|u(z)|^2
(1-|z|^2)^{\al}\dif\mu(z)
=\|u\|_{L^2(\bB_n,\mu_\al)}^2.
\qedhere
\end{align*}
\end{proof}

\subsection*{A factor to preserve the polyanalyticity when passing from the unit ball to the Siegel domain}

\begin{definition}\label{def:q}
Define $q_m\colon\bH_n\to\bC$,
\[
q_m(\xi)
\eqdef\left(\frac{1+\iu \conju{\xi_n}}%
{1-\iu\xi_n}\right)^{m-1}.
\]
\end{definition}

\begin{lemma}
\label{lem:composition_and_q}
Let $f\in\cA_m(\bB_n)$.
Then
$(f\circ\psi)\cdot q_m\in\cA_m(\bH_n)$.
\end{lemma}

\begin{proof}
This proof is similar to the proof of Lemma~\ref{lem:composition_by_Pessoa}.
The main idea is that the factor
$(1+\iu\overline{\xi_n})^{m-1}$,
appearing in the numerator of $q_m(\xi)$,
cancels the denominators of the expressions $\overline{\psi(\xi)}^{\,j}$, where $|j|<m$.
We represent $f$ in the form~\eqref{eq:polyanalytic_function_as_polynomial}, compose with $\psi$, and multiply by $q_m$:
\begin{align*}
u(\xi)
&\eqdef
(f\circ\psi)(\xi) q_m(\xi)
=
\sum_{|j|\le m}
h_j(\psi(\xi))%
\,\left(\prod_{s=1}^{n-1}
\frac{(2\iu\conju{\xi_s})^{j_s}}%
{(1+\iu\conju{\xi_n})^{j_s}}
\right)
\frac{(1-\iu\conju{\xi_n})^{j_n}}%
{(1+\iu\conju{\xi_n})^{j_n}}
\;
\frac{(1+\iu\conju{\xi_n})^{m-1}}%
{(1-\iu\xi_n)^{m-1}}
\\
&=
\sum_{|j|\le m}
\frac{h_j(\psi(\xi))}%
{(1-\iu\xi_n)^{m-1}}
\left(\prod_{s=1}^{n-1}
(2\iu\conju{\xi_s})^{j_s}\right)
(1-\iu \conju{\xi_n})^{j_n}
(1+\iu\conju{\xi_n})^{m-|j|-1}.
\end{align*}
For each $j$, the corresponding summand is the product of an analytic function
by a polynomial in $\conju{\xi_1},\ldots,\conju{\xi_n}$ of total degree $m-1$.
\end{proof}

\begin{remark}
Another way to prove Lemma~\ref{lem:composition_and_q},
computing $\conju{D}^j u$,
seems to be more complicated.
We will show it only for $n=2$ and $m=2$.
In this case,
\[
u(\xi)
=
\frac{1+\iu\conju{\xi_2}}%
{1-\iu\xi_2}
\,
f\left(
-\frac{2\iu\xi_1}{1-\iu\xi_2},
\frac{1+\iu\xi_2}{1-\iu\xi_2}
\right),
\]
By the well-known chain rule and product rule for Wirtinger derivatives,
\begin{align*}
(\conju{D}^{(1,0)}u)(\xi)
&=\frac{2\iu(\conju{D}^{(1,0)}f)(\psi(\xi))}{1-\iu\xi_2}, \\
(\conju{D}^{(0,1)}u)(\xi)&=\frac{\iu f(\psi(\xi))}{1-\iu\xi_2}\, + \, \frac{2\left(\left(\conju{\xi_1}\,\conju{D}^{(1,0)}-\iu\conju{D}^{(0,1)}\right) f \right)(\psi(\xi))}{(1-\iu\xi_2)(1+\iu\conju{\xi_2})}, \\ 
(\conju{D}^{(2,0)}u)(\xi)
&=\frac{-4(\conju{D}^{(2,0)}f)(\psi(\xi))}{(1-\iu\xi_2)(1+\iu\conju{\xi_2})}, \\
(\conju{D}^{(1,1)}u)(\xi)
&=\frac{4\left(\left(\iu\conju{\xi_1}\,\conju{D}^{(2,0)}+\conju{D}^{(1,1)}\right)f\right)(\psi(\xi))}{(1-\iu\xi_2)(1+\iu\conju{\xi_2})^2}, \\
(\conju{D}^{(0,2)}u)(\xi)
&=\frac{4\left(\left(\conju{\xi_2}^2\,\conju{D}^{(2,0)}
-2\iu\conju{\xi_1}\,\conju{D}^{(1,1)}
-\conju{D}^{(0,2)}\right)f\right)(\psi(\xi))}{(1-\iu\xi_2)(1+\iu\conju{\xi_2})^3}.
\end{align*}
Since $f\in\cA_2(\bB_n)$,
we conclude that $u\in\cA_2(\bH_n)$.
\end{remark}

\subsection*{A weighted change of variables which unitarily maps $\cA_m^2(\bB_n,\mu_\al)$
onto
$\cA_m^2(\bH_n,\nu_\al)$}

\begin{definition}
Define
$V\colon \cA_m^2(\bB_n,\mu_\al)\to \cA_m^2(\bH_n,\nu_\al)$ by
$Vu\eqdef (u\circ\psi)\cdot h_\al\cdot q_m$, i.e.,
\[
(Vu)(\xi)\eqdef u(\psi(\xi)) h_\al(\xi) q_m(\xi).
\]
\end{definition}

\begin{proposition}
\label{prop:V_is_well_defined_and_unitary}
$V$ is a well-defined unitary operator
$\cA_m^2(\bB_n,\mu_\al)\to\cA_m^2(\bH_n,\nu_\al)$.
\end{proposition}

\begin{proof}
Lemma~\ref{lem:composition_and_q}
assures that $Vu\in\cA_m(\bH_n)$
for every $u$ in $\cA_m^2(\bB_n,\mu_\al)$.
Lemma~\ref{lem:h_isometry},
combined with the identity
$|q_m(\xi)|=1$,
provides the isometric property of $V$.
It is easy to verify that the adjoint operator $V^\ast$ acts by
\begin{equation}\label{eq:Vast}
(V^*f)(z)
=\frac{f(\omega(z))}{h_\al(\omega(z))q_m(\omega(z))},
\end{equation}
and that $V^\ast$
is the inverse operator to $V$.
\end{proof}

\subsection*{Computation of the RK on the Siegel domain}

We define $t^\be$
via the principal argument of $t$,
see Remark~\ref{rem:def_power}.
The formulas $(tu)^\be=t^\be u^\be$ and $(t/u)^\be$ are not always true.
Let us recall some sufficient conditions for these formulas to be true.

\begin{lemma}\label{lem:power_of_product_or_quotient}
Let $t,u\in\bC\setminus\{0\}$ and $\be\in\bC$.
\begin{enumerate}
\item If $\Re(t)>0$ and $\Re(u)>0$,
then $(tu)^\be=t^\be u^\be$.
\item If $\Re(t)>0$ and $\Re(t/u)>0$,
then $(t/u)^\be=t^\be/u^\be$.
\end{enumerate}
\end{lemma}

\begin{proof}
1. The assumptions on $t$ and $u$
imply that $\arg(tu)=\arg(t)+\arg(u)$.\\
2. Follows from part 1 applied to $t$ and $u/t$.
\end{proof}

\begin{lemma}\label{lem:magic_power}
Let $\xi,\eta\in\bH_n$ and $\be\ge0$.
Then
\[
(1-\langle\psi(\xi),\psi(\eta)\rangle)^\be
=\frac{4^\be
\left(\frac{\xi_n-\conju{\eta_n}}{2\iu}-\langle \xi',\eta'\rangle\right)^\be}%
{(1-\iu\xi_n)^\be\,(1+\iu\conju{\eta_n})^\be}.
\]
\end{lemma}

\begin{proof}
Due to~\eqref{eq:psi_inner_prod},
$1-\langle\psi(\xi),\psi(\eta)\rangle=t/(uv)$,
where
\[
t\eqdef
4\left(\frac{\xi_n-\conju{\eta_n}}{2\iu}-\langle \xi',\eta'\rangle\right),\qquad
u\eqdef 1-\iu\xi_n,\qquad
v\eqdef 1+\iu\conju{\eta_n}.
\]
Since $|\psi(\xi)|<1$ and $|\psi(\eta)|<1$,
we obtain
\[
\Re(t/(uv))
=\Re(1-\langle\psi(\xi),\psi(\eta)\rangle)>0.
\]
Furthermore, $\Re(u)=1+\Im(\xi_n)>0$ and $\Re(v)=1+\Im(\eta_n)>0$.
So, by Lemma~\ref{lem:power_of_product_or_quotient},
\[
\left(\frac{t}{uv}\right)^\be
=\frac{t^\be}{(uv)^\be}
=\frac{t^\be}{u^\be v^\be}.
\qedhere
\]
\end{proof}

\begin{theorem}
\label{thm:RK_Siegel}
Let $n,m\in\bN$ and $\al>-1$.
Then for every $\xi$ in $\bH_n$,
the following function $\wt{K}_\xi$
is the reproducing kernel of $\cA_m^2(\bH_n,\nu_\al)$
at the point $\xi$:
\begin{equation}\label{eq:RK_Siegel}
\wt{K}_{\xi}(\eta)
=\frac{\left(\frac{\xi_n-\conju{\eta_n}}{2\iu}
-\langle \xi',\eta'\rangle\right)^{m-1}}%
{\left(\frac{\eta_n-\conju{\xi_n}}{2\iu}
-\langle \eta',\xi'\rangle\right)^{n+m+\al}}\;
R_{m-1}^{(\al,n-1)}(\rho_{\bH_n}(\xi,\eta)^2).
\end{equation}
\end{theorem}

\begin{proof}
Due to Proposition~\ref{prop:V_is_well_defined_and_unitary},
we can apply Proposition~\ref{prop:pusforward_kernel}
with $H_1=\cA_m(\bB_n,\mu_\al)$, $H_2=\cA_m(\bH_n,\nu_\al)$,
and $J(\xi)\eqdef h_\al(\xi)q_m(\xi)$.
So, for every $\xi$ in $\bH_n$,
the next function is the RK of $\cA_m(\bH_n,\nu_\al)$ associated to the point $\xi$:
\[
\wt{K}_\xi(\eta)
=\conju{h_\al(\xi)q_m(\xi)}
h_\al(\eta) q_m(\eta)
K_{\psi(\xi)}(\psi(\eta)).
\]
Substitute formula~\eqref{eq:RK_ball} for $K$:
\[
\wt{K}_\xi(\eta)
=\conju{h_\al(\xi)q_m(\xi)}
h_\al(\eta) q_m(\eta)\;
\frac{\left(1-\langle \psi(\xi),\psi(\eta) \rangle \right)^{m-1}}%
{\left( 1-\langle \psi(\eta),\psi(\xi) \rangle \right)^{n+m+\al}} 
\;R_{m-1}^{(\al,n-1)}(\rho_{\bH_n}(\xi,\eta)^2).
\]
Then, substitute
the definitions of $h_\al$, $q_m$
and use Lemma~\ref{lem:magic_power}:
\begin{align*}
\wt{K}_\xi(\eta)
&=
R_{m-1}^{(\al,n-1)}(\rho_{\bH_n}(\xi,\eta)^2)\;
\frac{2^{n+\al+1}(1+\iu\conju{\eta_n})^{m-1}}{(1-\iu\eta_n)^{n+m+\al}}\;
\frac{2^{n+\al+1}\,(1-\iu\xi_n)^{m-1}}%
{(1+\iu\conju{\xi_n})^{n+m+\al}}
\\
&\qquad\times
\frac{4^{m-1} \left(\frac{\xi_n-\conju{\eta_n}}{2\iu}-\langle \xi',\eta'\rangle\right)^{m-1}}%
{(1-\iu\xi_n)^{m-1}\,
(1+\iu\conju{\eta_n})^{m-1}}\;
\frac{(1+\iu\conju{\xi_n})^{n+m+\al}\,
(1-\iu\eta_n)^{n+m+\al}}%
{4^{n+m+\al}
\left(\frac{\eta_n-\conju{\xi_n}}{2\iu}
-\langle \eta',\xi'\rangle\right)^{n+m+\al}}.
\end{align*}
Simplifying this expression we obtain
the right-hand side of~\eqref{eq:RK_Siegel}.
\end{proof}

\begin{corollary}
\label{cor:RK_Siegel_norm}
Let $n,m\in\bN$ and $\al>-1$.
Then for every $\xi$ in $\bH_n$,
\begin{equation}\label{eq:Kz_Siegel_norm}
\|\wt{K}_\xi\|_{\cA_m^2(\bH_n,\nu_\al)}^2
=\wt{K}_\xi(\xi)
=\binom{n+m-1}{n}\,%
\frac{\Be(\al+1,n)}%
{\Be(\al+m,n)}\,
\frac{1}{(\Im(\xi_n)-|\xi'|^2)^{\al+n+1}}.
\end{equation}
\end{corollary}

\begin{remark}
Analogously to the case of the unit ball, we get some formulas equivalent  to~\eqref{eq:RK_Siegel}, using~\eqref{eq:rho_Siegel_efficient} and~\eqref{eq:Jacobi_explicit}:
\begin{align}
\label{eq:RK_Siegel_Jacobi}
\wt{K}_\xi(\eta)
&= \frac{\left(\frac{\xi_n-\conju{\eta_n}}{2\iu}
-\langle \xi',\eta'\rangle\right)^{m-1}}%
{\left(\frac{\eta_n-\conju{\xi_n}}{2\iu}
-\langle \eta',\xi'\rangle\right)^{n+m+\al}}\;%
\frac{(-1)^{m-1}\Be(\al+1,n)}{\Be(\al+m,n)}
P_{m-1}^{(\al,n)}(2\rho_{\bH_n}(\xi,\eta)^2-1)
\\[1.5ex]
\nonumber
&=\frac{\left(\frac{\xi_n-\conju{\eta_n}}{2\iu}
-\langle \xi',\eta'\rangle\right)^{m-1}}%
{\left(\frac{\eta_n-\conju{\xi_n}}{2\iu}
-\langle \eta',\xi'\rangle\right)^{n+m+\al}}\;
\frac{(-1)^{m-1}\Be(\al+1,n)}{\Be(\al+m,n)}
\times
\\
\label{eq:RK_Siegel_Jacobi_2}
&\qquad \times P_{m-1}^{(\al,n)}
\left(1-\frac{2(\Im(\xi_n)-|\xi'|^2)%
(\Im(\eta_n)-|\eta'|^2)}%
{\left|\frac{\xi_n-\conju{\eta_n}}{2\iu}%
-\langle \xi',\eta'\rangle\right|^2}\right)
\\[1.5ex]
\nonumber
&=
\frac{\left(\frac{\xi_n-\conju{\eta_n}}{2\iu}
-\langle \xi',\eta'\rangle\right)^{m-1}}%
{\left(\frac{\eta_n-\conju{\xi_n}}{2\iu}
-\langle \eta',\xi'\rangle\right)^{n+m+\al}}\;%
\frac{(-1)^{m-1}\,\Gamma(\al+1)}{\Gamma(\al+n+1)\,(m-1)!} \times
\\
\label{eq:RK_Siegel_explicit}
&\qquad \times\sum_{s=0}^{m-1}%
(-1)^s \binom{m-1}{s} \frac{\Gamma(\al+m+n+s)}{\Gamma(\al+s+1)}
\left(\frac{(\Im(\xi_n)-|\xi'|^2)%
(\Im(\eta_n)-|\eta'|^2)}%
{\left|\frac{\xi_n-\conju{\eta_n}}{2\iu}%
-\langle \xi',\eta'\rangle\right|^2}\right)^s.
\end{align}

\end{remark}

\begin{remark}
In the case $n=1$, i.e., for the upper halfplane $\bH_1$, formula~\eqref{eq:RK_Siegel} simplifies to
\begin{equation}\label{eq:RK_halfplane}
\wt{K}_\xi(\eta)
= \frac{\left(\frac{\xi-\conju{\eta}}{2\iu}\right)^{m-1}}%
{\left(\frac{\eta-\conju{\xi}}{2\iu}\right)^{m+\al+1}}
R_{m-1}^{(\al,0)}\left(\frac{|\xi-\eta|^2}{|\conju{\xi}-\eta|^2}\right) \qquad (\xi,\eta\in\bH_1).
\end{equation}
In particular, for $\al=0$, this expression coincides with formula~\cite[Corollary 2.5]{Pessoa2013} obtained by another method.
\end{remark}

\begin{remark}\label{rem:biholomophic_changes_of_variables_in_Siegel}
Generalizing ideas of this paper, it is possible to associate a unitary operator
(namely, a certain weighted shift)
in $\cA_m^2(\bH_n,\nu_\al)$ to every biholomorphism of the Siegel domain $\bH_n$.
In particular, using~\eqref{eq:RK_Siegel}, we have verified that the space $\cA_m(\bH_n,\nu_\al)$
is invariant under the unweighted changes of variables, corresponding to the quasi-parabolic, nilpotent, and quasi-nilpotent groups from~\cite[Section~3]{QuirogaVasilevski2007}.
\end{remark}

\bigskip

\section*{Acknowledgements}

The authors have been
partially supported by
Proyecto CONACYT
``Ciencia de Frontera'' FORDECYT-PRONACES/61517/2020,
by CONACYT (Mexico) scholarships,
and by SIP-IPN projects (Instituto Polit\'{e}cnico Nacional, Mexico).
We are grateful to professor Nikolai Vasilevski for inviting us to this area of mathematics and to professor Armando S\'{a}nchez-Nungaray for the term ``homogeneously polyanalytic function''.

\medskip\noindent
Christian Rene Leal-Pacheco\newline
Centro de Investigaci\'{o}n y de Estudios Avanzados
del Instituto Polit\'{e}cnico Nacional\newline
Departamento de Matem\'{a}ticas\newline
Apartado Postal 07360\newline
Ciudad de M\'{e}xico\newline
Mexico\newline
e-mail: christian.leal.pacheco@gmail.com\newline
\myurl{https://orcid.org/0000-0001-5738-4904}

\bigskip\noindent
Egor A. Maximenko\newline
Instituto Polit\'{e}cnico Nacional\newline
Escuela Superior de F\'{i}sica y Matem\'{a}ticas\newline
Apartado Postal 07730\newline
Ciudad de M\'{e}xico\newline
Mexico\newline
e-mail: egormaximenko@gmail.com\newline
\myurl{https://orcid.org/0000-0002-1497-4338}

\bigskip\noindent
Gerardo Ramos-Vazquez\newline
Centro de Investigaci\'{o}n y de Estudios Avanzados
del Instituto Polit\'{e}cnico Nacional\newline
Departamento de Matem\'{a}ticas\newline
Apartado Postal 07360\newline
Ciudad de M\'{e}xico\newline
Mexico\newline
e-mail: ger.ramosv@gmail.com\newline
\myurl{https://orcid.org/0000-0001-9363-8043}

\end{document}